\documentclass[11pt]{amsart}
\usepackage{amsmath}
\usepackage[english,  activeacute]{babel}
\usepackage[latin1]{inputenc}
\usepackage{amssymb}
\usepackage{amsthm}
\usepackage{graphics,graphicx}
\usepackage{array}
\usepackage{a4wide}
\allowdisplaybreaks
\usepackage{color, url}
\usepackage{float}

\theoremstyle{plain}
\newtheorem{theorem}{Theorem}

\newtheorem{corollary}[theorem]{Corollary}

\theoremstyle{definition}

\date{\today}

\title[Capacity statistic on compositions]{Further results for the capacity statistic distribution on compositions of $1$s and $2$s}

\begin{document}

\author[M. Shattuck]{Mark Shattuck}
\address{Department of Mathematics, University of Tennessee,
37996 Knoxville, TN}
\email{mshattuc@utk.edu}

\begin{abstract}
In this paper, we study additional aspects of the capacity distribution on the set $\mathcal{B}_n$ of compositions of $n$ consisting of $1$'s and $2$'s.  Among our results are further recurrences for this distribution as well as formulas for the total capacity and sign balance on $\mathcal{B}_n$.  We provide algebraic and combinatorial proofs of our results.  We also give combinatorial explanations of some prior results where such a proof was requested.  Finally, the joint distribution of the capacity statistic with two further parameters on $\mathcal{B}_n$ is briefly considered.
\end{abstract}

\maketitle

\section{Introduction}

A \emph{composition} of a positive integer $n$ is a sequence of positive integers, called \emph{parts}, whose sum is $n$.  A composition $\pi$ is often represented as a vector $(\pi_1,\ldots,\pi_r)$ whose components are the parts of $\pi$, or sequentially as $\pi_1\cdots \pi_r$, where the commas between the parts are omitted.  In the latter case, a sequence of $\ell$ consecutive parts of the same size $a$ is usually denoted by $a^\ell$.  Let $\mathcal{C}_n$ denote the set of compositions of $n$ for $n \geq 1$, with $\mathcal{C}_0$ consisting of the single empty composition with zero parts.  For example, if $n=4$, then
$$\mathcal{C}_4=\{(1,1,1,1),(1,1,2),(1,2,1),(1,3),(2,1,1),(2,2),(3,1),(4)\}.$$
It is well-known that $|\mathcal{C}_n|=2^{n-1}$ for all $n \geq 1$.

Given a word $w=w_1\cdots w_m$ on the alphabet of positive integers, the \emph{bargraph} of $w$, denoted by $b(w)$, consists of $m$ columns starting from a horizontal line each comprised of unit squares, called \emph{cells}, such that the $i$-th column from the left contains $w_i$ cells for $1 \leq i \leq m$.  Here, we will represent a composition $\pi=\pi_1\cdots\pi_r \in \mathcal{C}_n$ geometrically as a bargraph containing $r$ columns and $n$ cells altogether.  We will frequently identify a composition $\pi$ with its bargraph $b(w)$, and vice versa, and treat the two interchangeably when speaking of a statistic defined on either.

By a \emph{water cell} of a composition $\pi$, we mean a unit square lying outside of $b(w)$ that would ``hold water'' by virtue of its location, assuming the usual rules of water flow.  The \emph{capacity} of $\pi$, denoted by $\text{cap}(\pi)$, will refer to the number of its water cells.  See Figure \ref{fig1} below illustrating the bargraph of $\pi\in\mathcal{C}_{30}$, which has 15 parts and a capacity of 10. The capacity statistic has been studied on a variety of discrete structures, represented as bargraphs, such as $k$-ary words \cite{BBK1}, compositions \cite{BBK2,MS}, permutations \cite{BBKS} and set partitions \cite{MS}.

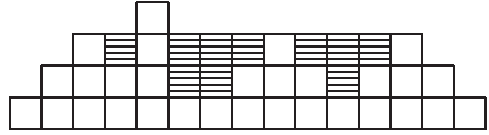
\begin{figure}[htp]
\begin{center}
\begin{picture}(200,65)
\setlength{\unitlength}{1.2pt}
\def\ccc{\put(0,0){\line(1,0){10}}\put(10,0){\line(0,1){10}}
\put(10,10){\line(-1,0){10}}\put(0,10){\line(0,-1){10}}}
\def\bbb{\multiput(0,0)(0,2){5}{\line(1,0){10}}}
\put(0,10){\multiput(0,0)(0,10){1}{\ccc}\multiput(10,0)(0,10){2}{\ccc}
\multiput(20,0)(0,10){3}{\ccc}\multiput(30,0)(0,10){2}{\ccc}
\multiput(40,0)(0,10){4}{\ccc}\multiput(50,0)(0,10){1}{\ccc}
\multiput(60,0)(0,10){1}{\ccc}\multiput(70,0)(0,10){2}{\ccc}
\multiput(80,0)(0,10){3}{\ccc}\multiput(90,0)(0,10){2}{\ccc}
\multiput(100,0)(0,10){1}{\ccc}\multiput(110,0)(0,10){3}{\ccc}
\multiput(120,0)(0,10){3}{\ccc}\multiput(130,0)(0,10){2}{\ccc}
\multiput(140,0)(0,10){1}{\ccc}
\multiput(30,20)(0,10){1}{\ccc\bbb}
\multiput(50,10)(0,10){2}{\ccc\bbb}\multiput(60,10)(0,10){2}{\ccc\bbb}
\multiput(70,20)(0,10){1}{\ccc\bbb}
\multiput(90,20)(0,10){1}{\ccc\bbb}\multiput(100,10)(0,10){2}{\ccc\bbb}
\multiput(110,20)(0,10){1}{\ccc\bbb}
}
\end{picture}
\caption{The bargraph of $\pi \in \mathcal{C}_{30}$ with cap($\pi$)=10, where the water cells are shaded.}\label{fig1}
\end{center}
\end{figure}

Let $\mathcal{B}_n$ denote the set of compositions of $n$ with parts in $\{1,2\}$ and let $f_n=f_{n-1}+f_{n-2}$ denote the $n$-th Fibonacci number, with $f_0=f_1=1$.  It is well-known (see, e.g., \cite[p.\,46, Exercise~14c]{Stan} that the cardinality of $\mathcal{B}_n$ is given by $f_n$ for all $n \geq0$.
In \cite{HT}, Hopkins and Tangboonduangjit considered the restriction of the capacity statistic to $\mathcal{B}_n$. Let $\mathcal{B}_{n,k}$ denote the subset of $\mathcal{B}_n$ whose members contain exactly $k$ water cells and let $w(n,k)=|\mathcal{B}_{n,k}|$. Illustrated in Figure \ref{fig2} below is the lone member of $\mathcal{B}_{5,1}$, which is seen to be the unique smallest composition with one water cell, and two members of $\mathcal{B}_{9,3}$, where the water cells are shaded.  Among the results in \cite{HT} are various recurrences for $w(n,k)$, expressions for $w(n,0)$ and an explicit as well as a generating function formula for the diagonal sum $\sum_{k=0}^nw(n-k,k)$.  In this paper, we study further aspects of the capacity distribution on $\mathcal{B}_n$ such as the average capacity and sign balance (i.e., the difference in cardinality of those members of $\mathcal{B}_n$ containing an even versus an odd number of water cells).  We also establish an explicit formula for $w(n,k)$ for $k \geq 1$ and provide requested combinatorial proofs of some prior results from \cite{HT}.

\begin{figure}[htp]
\begin{center}
\begin{picture}(275,65)
\setlength{\unitlength}{1.5pt}
\def\ccc{\put(0,0){\line(1,0){10}}\put(10,0){\line(0,1){10}}
\put(10,10){\line(-1,0){10}}\put(0,10){\line(0,-1){10}}}
\def\bbb{\multiput(0,0)(0,2){5}{\line(1,0){10}}}
\put(0,10){\multiput(0,0)(0,10){2}{\ccc}\multiput(10,0)(0,10){1}{\ccc}
\multiput(20,0)(0,10){2}{\ccc}\multiput(10,10)(0,10){1}{\ccc\bbb}
\multiput(50,0)(0,10){2}{\ccc}
\multiput(60,0)(0,10){1}{\ccc}\multiput(60,10)(0,10){1}{\ccc\bbb}
\multiput(70,0)(0,10){1}{\ccc}\multiput(70,10)(0,10){1}{\ccc\bbb}
\multiput(80,0)(0,10){2}{\ccc}
\multiput(90,0)(0,10){1}{\ccc}\multiput(90,10)(0,10){1}{\ccc\bbb}
\multiput(100,0)(0,10){2}{\ccc}
\multiput(130,0)(0,10){2}{\ccc}
\multiput(140,0)(0,10){2}{\ccc}
\multiput(150,0)(0,10){1}{\ccc}\multiput(150,10)(0,10){1}{\ccc\bbb}
\multiput(160,0)(0,10){1}{\ccc}\multiput(160,10)(0,10){1}{\ccc\bbb}
\multiput(170,0)(0,10){1}{\ccc}\multiput(170,10)(0,10){1}{\ccc\bbb}
\multiput(180,0)(0,10){2}{\ccc}
}
\end{picture}
\caption{The bargraphs of $212 \in\mathcal{B}_{5,1}$ and of $211212,\,221112 \in \mathcal{B}_{9,3}$.}\label{fig2}
\end{center}
\end{figure}
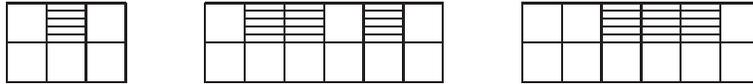

We will denote the capacity distribution on $\mathcal{B}_n$ by $b_n=b_n(y)$, i.e.,
$$b_n=\sum_{\pi \in \mathcal{B}_n}y^{\text{cap}(\pi)}=\sum_{k=0}^nw(n,k)y^k, \qquad n \geq1,$$
with $b_0=1$.  For example, if $n=6$, then
$$\mathcal{B}_6=\{1^6,1^42,1^321,1^221^2,1^22^2,121^3,1212,1221,21^4,21^22,2121,2^21^2,2^3\}$$
and $b_6=10+2y+y^2$.

The organization of this paper is as follows. In the second section, we consider various aspects of the distribution $b_n$.  In subsection 2.1, we provide combinatorial proofs of some prior formulas from \cite{HT} related to a class of restricted compositions whose cardinality coincides with the diagonal sum $\sum_{k=0}^nw(n-k,k)$.  We then determine explicit formulas for $w(n,k)$ where $k \geq 1$, as well as for $\frac{d}{dy}b_n(y)\mid_{y=1}$ and $b_n(-1)$, by use of generating functions.  In subsection 2.3, we  provide counting arguments for these formulas, where in the case of $b_n(-1)$, a certain sign-changing involution is defined on $\mathcal{B}_n$ in isolating a subset accounting for the explicit formula.  In  the final subsection, we establish both a third and a fourth order recurrence for $b_n$ by combinatorial arguments, where a creative use of the inclusion-exclusion principle is required for the latter.

In the third section, we consider the joint distribution of capacity with two other parameters on $\mathcal{B}_n$. One of these parameters records the number of $1$'s within a member of $\mathcal{B}_n$, whereas the other tracks the sum of the numbers covered by the left halves of dominoes when a member of $\mathcal{B}_n$ is represented as a square-and-domino tiling of length $n$.  An explicit formula is deduced for the generating function of this joint distribution on $\mathcal{B}_n$ and some special cases of the distribution are studied in greater detail.

\section{Properties of the capacity distribution on $\mathcal{B}_n$}

\subsection{Combinatorial proofs of prior formulas}

Let $\mathcal{D}_n$ denote the set of compositions $(a_1,\ldots,a_r)$ of a positive integer $n$ for which either $r=1,2$ or $r\geq 3$, with the parts $a_2,\ldots,a_{r-1}$ each even.  That is, any \emph{internal} part, if it exists, within a member of $\mathcal{D}_n$ is even.  Let $d(n)=|\mathcal{D}_n|$ for $n \geq 1$, with $d(0)=1$.  In \cite{HT}, it was shown $d(n)=\sum_{k=0}^nw(n-k,k)$ for $n \geq0$ and that $d(n)$ is equal to the sequence $A052955[n-1]$ in \cite{Sloane} for $n \geq 1$.  These combinatorial interpretations for A052955 in terms of restricted compositions and diagonal water cell sums are apparently new.

In \cite{HT}, it was requested to find direct combinatorial proofs of the following recurrences for $d(n)$:
\begin{equation}\label{d(n)0rec1}
d(n)=d(n-1)+2d(n-2)-2d(n-3), \qquad n \geq 4,
\end{equation}
and
\begin{equation}\label{d(n)0rec2}
d(n)=2d(n-2)+1, \qquad n \geq 3.
\end{equation}

We provide such proofs for \eqref{d(n)0rec1} and \eqref{d(n)0rec2} using the interpretation given above for $d(n)$ in terms of compositions. \medskip

\noindent \emph{Combinatorial proofs of recurrences \eqref{d(n)0rec1} and \eqref{d(n)0rec2}:} \medskip

We first show \eqref{d(n)0rec1} and let $n \geq 4$.  Upon adding 1 to the final part of an arbitrary member of $\mathcal{D}_{n-1}$, we have that there are $d(n-1)$ members of $\mathcal{D}_n$ whose final part is greater than 1.  Thus, by subtraction and since $n\geq 4$, there are $d(n-2)-d(n-3)$ members of $\mathcal{D}_{n-2}$ whose final part equals $1$, the subset of $\mathcal{D}_{n-2}$ of which we denote by $S$.  Using $S$, one can obtain the remaining unaccounted for members of $\mathcal{D}_n$, i.e., those whose last part equals 1, as follows.  Let $\rho \in S$.  Then either insert a part 2 directly prior to the terminal 1 of $\rho$, or add 2 to the penultimate part of $\rho$. Thus, each member of $S$ is seen to give rise to exactly two members of $\mathcal{D}_n$ whose final part is 1.  Further, $n \geq 4$ implies a member of $\mathcal{D}_n$ ending in 1 cannot end in $1,1$, and hence all such members of $\mathcal{D}_n$ are generated in a unique manner from those in $S$ as described.  Thus, there are $2(d(n-2)-d(n-3))$ members of $\mathcal{D}_n$ whose final part is 1, and combining with the prior case completes the proof of \eqref{d(n)0rec1}.

For \eqref{d(n)0rec2}, given $\pi\in  \mathcal{D}_{n-2}$, where $n \geq 3$, let $z$ denote the final part of $\pi$.  If $z$ is even, then either replace $z$ with $z+2$ or append a 2 to the end of $\pi$.  If $z$ is odd, then replace $z$ with either $z+2$ or $z+1,1$.  Note that these operations give rise to almost all of the members of $\mathcal{D}_n$.  If $n$ is odd, then the two-part composition $(n-2,2)$ is missed by the operations described above, whereas if $n$ is even, then $(n-1,1)$ is missed.  Further, one may verify that no other members of $\mathcal{D}_n$ are missed and that all other members of $\mathcal{D}_n$ arise uniquely by applying one of the operations described above to some $\pi \in \mathcal{D}_{n-2}$.  Since exactly two operations are applied to each member of $\mathcal{D}_{n-2}$, recurrence \eqref{d(n)0rec2} follows. \hfill \qed \medskip

It was also shown in \cite{HT} using Riordan arrays that
\begin{equation}\label{d(n)gf}
\sum_{n\geq0}d(n)x^n=\frac{1-x^2+x^3}{1-x-2x^2+2x^3},
\end{equation}
and a direct combinatorial explanation of \eqref{d(n)gf} was requested.  We complete this section by providing such an explanation.  \medskip

\noindent \emph{Combinatorial proof of formula \eqref{d(n)gf}:} \medskip

Let $\mathcal{D}_n^{(e)}$ and $\mathcal{D}_n^{(o)}$ denote the subsets of $\mathcal{D}_n$ whose members each contain at least two parts and end in an even or an odd part, respectively.  We first determine the generating function for the sequence $\{|D_n^{(e)}|\}_{n\geq 3}$.  To do so, note that if $n$ is even with $n=2m$ for some $m\geq 2$, then $\pi \in \mathcal{D}_{n}^{(e)}$ implies $\pi$ can be expressed as $\pi=p\pi'$, where $p$ is even and hence $p=2r$ with $r \in [m-1]$, and $\pi'$ is such that $\frac{1}{2}\pi'$ is any composition of the positive integer $m-r$. Thus, there are $2^{m-r-1}$ possibilities for $\pi'$ given $p$.  Similarly, if $n$ is odd with $n=2m-1$ for some $m \geq 2$, then $\pi \in \mathcal{D}_n^{(e)}$ implies $\pi=p\pi'$, where now $p=2r-1$ for some $r \in [m-1]$ and $\frac{1}{2}\pi'$ is a composition of $m-r$.  Combining the even and odd cases for $n$, we have
$$\sum_{n\geq 3}|\mathcal{D}_n^{(e)}|x^n=\frac{x}{1-x}\cdot \frac{x^2}{1-2x^2}=\frac{x^3}{(1-x)(1-2x^2)},$$
where $\frac{x}{1-x}$ accounts for the part $p$ in the decompositions described above and $\frac{x}{1-2x}$ accounts for the composition $\frac{1}{2}\pi'$ of $m-\lceil p/2\rceil$ for some $m>\lceil p/2\rceil$ (which is then doubled to yield $\pi'$). Note that the nonempty composition $\frac{1}{2}\pi'$ can be chosen independently of the part $p$, which justifies the multiplication of the two constituent factors in the generating function.

A similar approach may be applied in finding $\sum_{n\geq 2}|\mathcal{D}_n^{(o)}|x^n$.  In this case, one may first form $\pi \in \mathcal{D}_{n+1}^{(e)}$ for some $n \geq 2$ and then subtract 1 from the final part of $\pi$ to obtain a member of $\mathcal{D}_n^{(o)}$, with all members of $\mathcal{D}_n^{(o)}$ seen to arise in this way.  This has the effect of dividing the generating function for the sequence $|\mathcal{D}_n^{(e)}|$ by $x$, and hence
$$\sum_{n\geq 2}|\mathcal{D}_n^{(o)}|x^n=\frac{x^2}{(1-x)(1-2x^2)}.$$
Finally, adding the contribution of $\frac{1}{1-x}$ coming from one-part (as well as the empty) compositions, we then get
$$\sum_{n\geq0}d(n)x^n=\frac{1}{1-x}+\frac{x^2(1+x)}{(1-x)(1-2x^2)}=\frac{1-x^2+x^3}{1-x-2x^2+2x^3},$$
as desired.  \hfill \qed

\subsection{Additional results for the capacity distribution}

In this section, we find explicit formulas for $w(n,k)$ for $k\geq1$, the total capacity taken over all members of $\mathcal{B}_n$ and the sign balance of the capacity statistic on $\mathcal{B}_n$.  To do so, first recall from \cite{HT} the generating function formulas
\begin{equation}\label{unif1}
\sum_{n\geq k}w(n,k)x^n=\frac{x^{k+4}}{(1-x)^2(1-x^2)^{k+1}}, \qquad k \geq 1,
\end{equation}
and
\begin{equation}\label{unif2}
\sum_{n\geq0}w(n,0)x^n=\frac{1-x+x^3}{(1-x)^2(1-x^2)}.
\end{equation}

Define the bivariate generating function
$$\sum_{n\geq0}\sum_{k=0}^nw(n,k)x^ny^k.$$
By \eqref{unif1} and \eqref{unif2}, we have
\begin{align}
F(x,y)&=\sum_{n\geq0}w(n,0)x^n+\sum_{k\geq1}y^k\sum_{n\geq k+4}w(n,k)x^n\notag\\
&=\frac{1-x+x^3}{(1-x)^2(1-x^2)}+\sum_{k\geq1}\frac{x^{k+4}y^k}{(1-x)^2(1-x^2)^{k+1}}\notag\\
&=\frac{1-x+x^3}{(1-x)^2(1-x^2)}+\frac{x^5y}{(1-x)^2(1-x^2)(1-xy-x^2)}\notag\\
&=\frac{1-x(1+y)-x^2(1-y)+2x^3-x^4y-x^5(1-y)}{(1-x)^2(1-x^2)(1-xy-x^2)}\notag\\
&=\frac{1-x(1+y)+x^2y+x^3(1-y)}{(1-x)^2(1-xy-x^2)}.\label{Fxy}
\end{align}
Taking $y=1$ in \eqref{Fxy} gives $F(x,1)=\frac{1}{1-x-x^2}=\sum_{n\geq0}f_nx^n$, as required.  Note that \eqref{Fxy} also follows as a special case of a more general result \cite[Theorem~2]{MS} concerning  compositions whose parts lie in $[p]$ for any $p\geq1$.

We now determine an explicit formula for $w(n,k)$ for all $k \geq1$, a formula for $w(n,0)$ having already been given in \cite{HT}.

\begin{theorem}\label{th1}
If $n \geq 5$ and $1 \leq k \leq n-4$, then
\begin{equation}\label{wnkform}
w(n,k)=\sum_{r=1}^{\lfloor\frac{n-k-2}{2}\rfloor}(n-k-2r-1)\binom{k+r-1}{k}.
\end{equation}
\end{theorem}
\begin{proof}
From the derivation above of the expression for $F(x,y)$, we have
$$\sum_{n\geq5}\sum_{k=1}^{n-4}w(n,k)x^ny^k=\frac{x^5y}{(1-x)^2(1-x^2)(1-xy-x^2)}.$$
To expand the preceding generating function, first note
\begin{align*}
\frac{x^5y}{(1-x^2)(1-xy-x^2)}&=\frac{x^4}{1-xy-x^2}-\frac{x^4}{1-x^2}=\sum_{r \geq 2}x^{2r}\left(\frac{1}{(1-xy)^{r-1}}-1\right)\\
&=\sum_{r\geq1}x^{2r+2}\sum_{k\geq1}\binom{k+r-1}{r-1}(xy)^k=\sum_{k\geq1}y^k\sum_{r\geq1}\binom{k+r-1}{k}x^{k+2r+2}.
\end{align*}
Hence, we have
\begin{align*}
\frac{x^5y}{(1-x)^2(1-x^2)(1-xy-x^2)}&=\sum_{k\geq1}y^k\sum_{r\geq1}\binom{k+r-1}{k}x^{k+2r+2}\sum_{n\geq1}nx^{n-1}\\
&=\sum_{k\geq1}y^k\sum_{r\geq1}\binom{k+r-1}{k}\sum_{n\geq k+2r+2}(n-k-2r-1)x^n\\
&=\sum_{n\geq5}\sum_{k=1}^{n-4}x^ny^k\sum_{r=1}^{\lfloor\frac{n-k-2}{2}\rfloor}(n-k-2r-1)\binom{k+r-1}{k},
\end{align*}
which implies \eqref{wnkform}.
\end{proof}

\begin{theorem}\label{totcapth}
If $n\geq0$, then the sum of the capacity values of all the members of $\mathcal{B}_n$ is given by
\begin{equation}\label{totcap}
\sum_{k=1}^nkw(n,k)=\frac{(n+1)L_{n+1}-f_n}{5}-2f_{n+1}+n+2,
\end{equation}
where $L_m=f_{m}+f_{m-2}$ for $m\geq 1$ denotes the $m$-th Lucas number.
\end{theorem}
\begin{proof}
By \eqref{Fxy}, we have
$$\frac{\partial}{\partial y}F(x,y)\mid_{y=1}=\frac{x}{(1-x-x^2)^2}-\frac{x(1-x+x^2)}{(1-x)^2(1-x-x^2)}=\frac{x^5}{(1-x)^2(1-x-x^2)^2}.$$
Now observe
$$[x^n]\left(\frac{1}{(1-x)(1-x-x^2)}\right)=\sum_{i=0}^nf_i=f_{n+2}-1, \qquad n \geq0,$$
where in the second equality, we have used \cite[Identity~1]{BQ}.
Thus, we have
$$[x^n]\left(\frac{1}{(1-x)^2(1-x-x^2)^2}\right)=\sum_{j=0}^n(f_{j+2}-1)(f_{n-j+2}-1),$$
upon considering the convolution of the generating function $\frac{1}{(1-x)(1-x-x^2)}$ with itself.  Note that the lower and upper indices of summation in the preceding sum may be replaced by $j=-2$ and $j=n+2$, respectively.  Replacing $j$ by $j-2$ in the sum then gives
\begin{align}
[x^n]\left(\frac{1}{(1-x)^2(1-x-x^2)^2}\right)&=\sum_{j=0}^{n+4}(f_{j}-1)(f_{n-j+4}-1)\notag\\
&=\sum_{j=0}^{n+4}f_jf_{n-j+4}-2\sum_{j=0}^{n+4}f_j+\sum_{j=0}^{n+4}1\notag\\
&=\frac{(n+6)L_{n+6}-f_{n+5}}{5}-2(f_{n+6}-1)+n+5, \qquad n \geq 0, \label{n+5tot}
\end{align}
where we have used Identities 1 and 58 from \cite{BQ} to obtain \eqref{n+5tot} from the equality preceding it.  If $n \geq 5$, then we have
$$[x^n]\frac{\partial}{\partial y}F(x,y)\mid_{y=1}=[x^{n-5}]\left(\frac{1}{(1-x)^2(1-x-x^2)^2}\right)=\frac{(n+1)L_{n+1}-f_n}{5}-2f_{n+1}+n+2,$$
by \eqref{n+5tot}, which implies \eqref{totcap} for $n\geq 5$.  Since both sides of \eqref{totcap} are seen to equal zero for $0 \leq n \leq 4$, the proof of \eqref{totcap} is complete.
\end{proof}

\noindent\textbf{Remarks:} Dividing the expression in \eqref{totcap} by $f_n$ yields the average capacity of all the members of $\mathcal{B}_n$.  Since $f_n\sim\frac{1}{\sqrt{5}}\left(\frac{1+\sqrt{5}}{2}\right)^{n+1}$ for large $n$, one obtains that this average is asymptotically equal to $\frac{n}{\sqrt{5}}$.  This result may be interpreted probabilistically in terms of tilings as follows.  Let $\mathcal{L}_n$ denote the set of linear square-and-domino tilings of length $n$; see, e.g., \cite[Chapter~1]{BQ}.  Then the preceding implies that if a member $\lambda \in \mathcal{L}_n$ and an element $x \in [n]$ are chosen at random where $n$ is large, then the probability that $x$ is covered by a square in $\lambda$ such that at least one domino occurs somewhere both to the left and to the right of $x$ is approximately $\frac{1}{\sqrt{5}}\approx0.447$. \medskip

Let $\mathcal{B}_n^{(e)}$ and $\mathcal{B}_n^{(o)}$ denote the subsets of $\mathcal{B}_n$ consisting of those members whose capacity is even or odd, respectively.   We have the following sign balance result for the capacity statistic on $\mathcal{B}_n$.

\begin{theorem}\label{capbalance}
If $n \geq0$, then
\begin{equation}\label{capbalancee1}
|\mathcal{B}_n^{(e)}|-|\mathcal{B}_n^{(o)}|=\sum_{k=0}^n(-1)^kw(n,k)=2n-4+(-1)^nf_{n-6},
\end{equation}
where $f_{-m}$ for $m\geq2$ is defined as $(-1)^mf_{m-2}$, with $f_{-1}=0$.
\end{theorem}
\begin{proof}
The formula is readily verified for $0 \leq n \leq 2$. Taking $y=-1$ in \eqref{Fxy} gives $F(x,-1)=\frac{1-x^2+2x^3}{(1-x)^2(1+x-x^2)}$.  Note
$$\frac{1}{(1-x)^2(1+x-x^2)}=\frac{1}{(1-x)^2}-\frac{x}{1-x}+\frac{x^2}{1+x-x^2},$$
and hence
$$[x^n]\left(\frac{1}{(1-x)^2(1+x-x^2)}\right)=n+(-1)^nf_{n-2}, \qquad n \geq0.$$
This implies for $n \geq 3$,
\begin{align*}
[x^n]F(x,-1)&=n+(-1)^nf_{n-2}-(n-2+(-1)^{n-2}f_{n-4})+2(n-3+(-1)^{n-3}f_{n-5})\\
&=2n-4+(-1)^n(f_{n-2}-f_{n-4}-2f_{n-5})\\
&=2n-4+(-1)^nf_{n-6},
\end{align*}
as desired.
\end{proof}

As a consequence of the preceding results, one obtains the following binomial coefficient identities.

\begin{corollary}\label{binomcor}
If $n\geq 5$, then
\begin{align}
\sum_{r=1}^m\sum_{k=2r+2}^{n-1}(n-k)\binom{k-r-2}{r-1}&=f_n-1-\lfloor n^2/4 \rfloor, \label{binomcore1}\\
\sum_{r=1}^m\sum_{k=2r+2}^{n-1}(n-k)(k-r-2)\binom{k-r-3}{r-1}&=\frac{(n+1)L_{n+1}-f_n}{5}-2f_{n+1}+n+2, \label{binomcore2}\\
\sum_{r=1}^m\sum_{k=2r+2}^{n-1}(-1)^{k-1}(n-k)\binom{k-r-2}{r-1}&=2n-5-\lfloor n^2/4 \rfloor+(-1)^nf_{n-6},\label{binomcore3}
\end{align}
where $m=\lfloor(n-3)/2\rfloor$.
\end{corollary}
\begin{proof}
First note $\sum_{k=0}^nw(n,k)=f_n$, $\sum_{k=1}^nkw(n,k)=\text{tot}_n(\text{cap})$ and $\sum_{k=0}^n(-1)^kw(n,k)=|\mathcal{B}_n^{(e)}|-|\mathcal{B}_n^{(o)}|$ for $n \geq0$, where $\text{tot}_n(\text{cap})$ denotes the sum of the capacity values of all the members of $\mathcal{B}_n$.
By \eqref{wnkform}, \eqref{totcap} and \eqref{capbalancee1}, we then obtain the following identities for $n \geq 5$:
\begin{align*}
\sum_{k=1}^{n-4}\sum_{r=1}^{p}(n-k-2r-1)\binom{k+r-1}{k}&=f_n-w(n,0),\\
\sum_{k=1}^{n-4}\sum_{r=1}^{p}k(n-k-2r-1)\binom{k+r-1}{k}&=\frac{(n+1)L_{n+1}-f_n}{5}-2f_{n+1}+n+2,\\
\sum_{k=1}^{n-4}\sum_{r=1}^{p}(-1)^k(n-k-2r-1)\binom{k+r-1}{k}&=2n-4+(-1)^nf_{n-6}-w(n,0),
\end{align*}
where $p=\lfloor (n-k-2)/2 \rfloor$ and $w(n,0)=1+\lfloor n^2/4 \rfloor$.  Interchanging summation, replacing the index $k$ with $k-2r-1$ and noting $(k-2r-1)\binom{k-r-2}{r-1}=(k-r-2)\binom{k-r-3}{r-1}$ in the second resulting identity yields \eqref{binomcore1}--\eqref{binomcore3}.
\end{proof}

\subsection{Combinatorial proofs}

In this section, we provide combinatorial explanations of the formulas in Theorems \ref{th1}--\ref{capbalance}, which were found by use of the generating function \eqref{Fxy}. \medskip

\noindent\emph{Proof of Theorem \ref{th1}:}\medskip

Note that $\pi \in \mathcal{B}_{n,k}$ for some $n \geq 5$ and $k \in [n-4]$ implies $\pi$ can be decomposed as
\begin{equation}\label{bnkform}
\pi=1^x2^{a_0}1\cdots 2^{a_{k-1}}12^{a_k}1^y,
\end{equation}
where $a_0,a_{k}\geq 1$ and $x,y,a_1,\ldots,a_{k-1} \geq0$ such that $x+y+2\sum_{i=0}^ka_i=n-k$. We first enumerate members of $\mathcal{B}_{n,k}$  of the form \eqref{bnkform}, where $\sum_{i=0}^ka_i=r+1$ for some $r \geq 1$.  Note that $x,y\geq0$ implies $2(r+1)\leq n-k$, i.e., $r \leq \lfloor(n-k-2)/2\rfloor$.  Then $a_0,a_{k+1}\geq1$ implies there are $\binom{k+r-1}{k}$ possibilities for $(a_0,\ldots,a_{k})$ for each $r$, by \cite[p.\,15]{Stan}.  Then $x+y=n-k-2r-2$ with $x,y \geq0$ yields $n-k-2r-1$ possibilities for $x$ and $y$ independent of the choice of the $a_i$.  Thus, there are $(n-k-2r-1)\binom{k+r-1}{k}$ members of $\mathcal{B}_{n,k}$ such that $\sum_{i=0}^ka_i=r+1$ for each $r$ and summing over all possible $r$ yields \eqref{wnkform}.
\hfill \qed \medskip

\noindent\emph{Proof of Theorem \ref{totcapth}:}\medskip

Formula \eqref{totcap} is easily seen to hold for $0 \leq n \leq 2$, so we may assume $n \geq 3$.  Let $\mathcal{B}_n'$ denote the set of marked members of $\mathcal{B}_n$ wherein a single part 1 corresponding to a water cell, if it exists, is marked.  Equivalently, we determine the cardinality of $\mathcal{B}_n'$.  Let $\mathcal{B}_n^*$ denote the set of marked members of $\mathcal{B}_n$ wherein any part $1$ may be marked.  Note that there are $f_if_{n-i-1}$ members $\pi \in \mathcal{B}_n^*$ of the form $\pi=\pi'\underline{1}\pi''$ such that $\pi' \in \mathcal{B}_i$, $\pi'' \in \mathcal{B}_{n-i-1}$ and the marked $1$ is underlined.  Summing over all $0 \leq i \leq n-1$ implies
$$|\mathcal{B}_n^*|=\sum_{i=0}^{n-1}f_if_{n-i-1}=\frac{(n+1)L_{n+1}-f_n}{5},$$
where in the second equality, we have used \cite[Identity~58]{BQ}, which was provided a combinatorial proof by the authors in terms of linear and circular tilings.

To complete our enumeration of $\mathcal{B}_n'$, we subtract from $\mathcal{B}_n^*$ the cardinality of those members in which the marked $1$ occurs as part of an initial or terminal run of $1$'s, and hence does not correspond to a water cell.  Note that there are $f_{n-\ell}-1$ members of $\mathcal{B}_n^*$ of the form $1^{\ell-1}\underline{1}\rho'$ for some $1 \leq \ell \leq n-2$, where the marked $1$ is underlined and $\rho'$ contains at least one $2$. Considering all $\ell$ gives $\sum_{\ell=1}^{n-2}(f_{n-\ell}-1)=f_{n+1}-n-1$ members of $\mathcal{B}_n^*$ (not all $1$'s) in which the marked $1$ occurs as part of an initial run of $1$'s, where we have made use of \cite[Identity~1]{BQ}, which may be readily explained bijectively.  By symmetry, there are the same number of members of $\mathcal{B}_n^*$ in which the marked $1$ occurs as part of a terminal run (preceded by a $2$).  Finally, there are $n$ members of $\mathcal{B}_n^*$ corresponding to the all $1$'s composition.  Combining the prior cases, we have
$$|\mathcal{B}_n^*-\mathcal{B}_n'|=2(f_{n+1}-n-1)+n=2f_{n+1}-n-2.$$
Subtracting this expression from the one found above for $|\mathcal{B}_n^*|$ gives the desired formula for $|\mathcal{B}_n'|$ and completes the proof. \hfill \qed \medskip

\noindent\emph{Proof of Theorem \ref{capbalance}:}\medskip

Let $\mathcal{K}_n$ denote the subset of $\mathcal{B}_n$ whose members contain at least two parts $2$ such that at least one part occurs between the first and last $2$'s.  Note that $\mathcal{K}_n=\varnothing$ if $0 \leq n \leq 4$ and formula \eqref{capbalancee1} is readily verified in these cases.  So assume $n \geq 5$ and we seek an involution on $\mathcal{K}_n$. Note that $\pi \in \mathcal{K}_n$ for $n \geq 5$ implies it may be decomposed as
\begin{equation}\label{K-ndec}
\pi=1^x2\alpha21^y, \quad x,y \geq0.
\end{equation}
Then $\alpha \in \mathcal{B}_{n-x-y-4}$ nonempty implies $x+y \leq n-5$. Define the sign of $\pi$ by $\text{sgn}(\pi)=(-1)^{\text{cap}(\pi)}$.  Note that $\text{cap}(\pi)=n-x-y-2\mu(\pi)$, where $\mu(\pi)$ denotes the number of $2$'s of $\pi$, and hence $\text{sgn}(\pi)=(-1)^{n-x-y}$.

We define a preliminary involution on $\mathcal{K}_n$ as follows.  If $x \geq 1$ and $\alpha=\alpha'1$ in \eqref{K-ndec}, then replace $\pi$ with the composition $1^{x-1}2(\alpha'2)21^y$, whereas if $x \geq 0$ and $x=\alpha'2$, then replace $\pi$ with $1^{x+1}2(\alpha'1)21^y$.  This pair of operations defines a sign-changing involution on $\mathcal{K}_n$ whose set of survivors, which we'll denote by $\mathcal{K}_n'$, consists of those $\pi$ expressible as $\pi=2(\alpha'1)21^y$, where $\alpha'$ is possibly empty.  If $n=5$, then $\mathcal{K}_n'$ contains a single member with negative sign, and \eqref{capbalancee1} is seen to hold when $n=5$ since $|\mathcal{B}_5-\mathcal{K}_5|=7$.

So assume $n \geq 6$ and considering the final part of the section $\alpha'$ within a member of $\mathcal{K}_n'$ leads to the pairings
$$2(\alpha''21)21^y\leftrightarrow 2(\alpha''1^2)21^{y+1}, \qquad y \geq0.$$
Thus, each member of $\mathcal{K}_n'$ is paired with another of opposite sign except for $\rho=2121^{n-5}$ or $\pi$ of the form $\pi=2(\alpha''1^2)2$, where $\alpha'' \in \mathcal{B}_{n-6}$. Note that $\rho$ has negative sign, whereas each $\pi$ of the stated form has sign $(-1)^n$.  Thus, the sum of the signs of all members of $\mathcal{K}_n'$, and hence also of $\mathcal{K}_n$, is given by $-1+(-1)^nf_{n-6}$.

To this, we add the contribution of members of $\mathcal{B}_n-\mathcal{K}_n$, which consists of compositions of $1$'s and $2$'s containing at most one $2$ or exactly two $2$'s with no part between them.  Since $\mathcal{B}_n-\mathcal{K}_n$ is a subset of $\mathcal{B}_{n,0}$, each of its members has positive sign, with $|\mathcal{B}_n-\mathcal{K}_n|=2n-3$, as there are $n$ compositions of $n$ containing at most one $2$ and $n-3$ of the form $1^x2^21^y$ for some $x,y\geq0$.  Adding $2n-3$ to the contribution towards the sign balance found above for $\mathcal{K}_n$ yields $2n-4+(-1)^nf_{n-6}$ for all $n \geq 6$, as desired.  \hfill \qed \medskip

\subsection{Recurrences for the capacity distribution}

Recall $b_n=b_n(y)$ for $n \geq1$ be given by $b_n=\sum_{\pi \in \mathcal{B}_n}y^{\text{cap}(\pi)}$, with $b_0=1$.  There are the following recurrence formulas for $b_n$.

\begin{theorem}\label{bnrec}
We have
\begin{equation}\label{bnrece1}
b_n=(1+y)b_{n-1}+(1-y)b_{n-2}-b_{n-3}+1-y,\qquad n \geq 3,
\end{equation}
and
\begin{equation}\label{bnrece2}
b_n=(2+y)b_{n-1}-2yb_{n-2}-(2-y)b_{n-3}+b_{n-4},\qquad n \geq 4,
\end{equation}
with $b_0=b_1=1$, $b_2=2$ and $b_3=3$.
\end{theorem}
\begin{proof}
We first show \eqref{bnrece1}, where the initial values for $0 \leq n \leq 2$ are clear.  Let $n \geq 3$ and note that appending a 1 to an arbitrary member of $\mathcal{B}_{n-1}$ does not change the capacity, whence the weight of all members of $\mathcal{B}_n$ ending in 1 is given by $b_{n-1}$.  On the other hand, inserting a 1 directly prior to the final 2 within a member of $\mathcal{B}_{n-1}$ ending in 2, except for $1^{n-3}2$, is seen to increase the capacity by one.  Thus, by subtraction, the contribution towards $b_n$ of the members of $\mathcal{B}_{n}$ ending in $1,2$ is given by $y(b_{n-1}-b_{n-2}-1)+1$.  Note that the $+1$ term accounts for the composition $1^{n-2}2$ of capacity zero, which does not arise by inserting 1 as described.  Finally, by subtraction, the weight of all members of $\mathcal{B}_{n-2}$ ending in 2 is given by $b_{n-2}-b_{n-3}$ and appending a 2 to such members does not change the capacity.  Thus, the contribution towards $b_n$ of all the members of $\mathcal{B}_n$ ending in $2,2$ is given by $b_{n-2}-b_{n-3}$ and combining with the prior cases implies \eqref{bnrece1}.

Replacing $n$ with $n-1$ in $b_n-(1+y)b_{n-1}-(1-y)b_{n-2}+b_{n-3}=1-y$ for $n \geq 4$, and equating expressions for $1-y$, leads to \eqref{bnrece2}.
\end{proof}

Note that recurrence \eqref{bnrece2} follows from the generating function formula \eqref{Fxy} above as well.  It is also possible to provide a counting argument for \eqref{bnrece2} as follows.  \medskip

\noindent\emph{Combinatorial proof of recurrence \eqref{bnrece2}:}\medskip

Let $n \geq 4$, and first observe $\mathcal{B}_n-\{1^{n-2}2\}$ is a (non-disjoint) union of the following three sets:
\begin{align*}
U=\{&\pi \in \mathcal{B}_n: \pi \text{ ends in } 1\},\\
V=\{&\pi \in \mathcal{B}_n: \pi \text{ contains at least two 2's and at least one 1 occurs between the rightmost}\\
&\text{two 2's}\},\\
W=\{&\pi \in \mathcal{B}_n: \pi \text{ contains at least two 2's, with the rightmost two 2's adjacent}\}.
\end{align*}
Let $|S|$ denote here $\sum_{\pi \in S}y^{\text{cap}(\pi)}$ for a subset $S$ of $\mathcal{B}_n$; i.e., $|S|$ gives the cardinality of the set of colored members of $S$  wherein each water cell receives one of $y$ colors in the case when $y$ is a fixed positive integer.  Applying the principle of inclusion-exclusion, and noting $V\cap W =\varnothing$, we have
\begin{equation}\label{incexcl}
b_n-1=|U|+|V|+|W|-|U\cap V|-|U\cap W|.
\end{equation}

To complete the proof of \eqref{bnrece2}, we seek expressions for the quantities on the right-hand side of \eqref{incexcl} in terms of $b_n$. Clearly, we have $|U|=b_{n-1}$.  Also, members of $V$ may be formed in a unique manner from members of $\mathcal{B}_{n-1}$ containing at least two $2$'s by inserting a single 1 directly prior to the rightmost 2.  This is seen to increase the capacity by one and hence $|V|=y(b_{n-1}-n+1)$, upon excluding the $n-1$ members of $\mathcal{B}_{n-1}$ that contain at most one $2$. To determine $|W|$, we consider the auxiliary set $X$ given by $X=\{\pi \in \mathcal{B}_n: \pi \text{ ends in 2,\,1}\}$. To form members of $W$ or $X$, we start with a composition $\pi \in \mathcal{B}_{n-1}-\{1^{n-1}\}$ and consider the final part of $\pi$.  If $\pi$ ends in 1, then we change the 1 directly following the rightmost 2 in $\pi$ to a 2, which is seen to yield all members of $W$.  If $\pi$ ends in a 2, then we append 1 to $\pi$, which yields the members of $X$.  Since the capacity is unchanged by the operation in either case, we have $b_{n-1}-1=|W|+|X|$.

To determine $|X|$, we consider the antepenultimate part of $\pi \in X$.  If $\pi=\pi'2^21$, then $\pi$ may be obtained by appending $2,1$ to a member of $\mathcal{B}_{n-3}$ ending in 2, of which there are $b_{n-3}-b_{n-4}$ possibilities, by subtraction.  If $\pi=\pi'121$ and $\pi \neq 1^{n-3}21$, then $\pi$ may be obtained from $\rho \in \mathcal{B}_{n-2}$ ending in 2, and not equal $1^{n-4}2$, by inserting a 1 both directly before and after the terminal 2 of $\rho$.  Since $\rho \neq 1^{n-4}2$, this operation always increases the capacity by one and hence there are $y(b_{n-2}-b_{n-3}-1)$ possibilities for such $\pi$, upon excluding the members of $\mathcal{B}_{n-2}$ ending in 1 as well as the single composition $1^{n-4}2$ from consideration for $\rho$.  Adding to this the contribution of $1^{n-3}21$, which has capacity zero, implies members of $X$ with antepenultimate part 1 have weight $y(b_{n-2}-b_{n-3})+1-y$.  Combining with the previous case implies
$|X|=yb_{n-2}+(1-y)b_{n-3}-b_{n-4}+1-y$, and hence
$$|W|=b_{n-1}-1-|X|=b_{n-1}-yb_{n-2}-(1-y)b_{n-3}+b_{n-4}-2+y.$$

To find $|U\cap V|$, first note that $\pi \in U \cap V$ implies $\pi$ is expressible as $\pi=\pi'21^x21^y$ for some $x,y\geq 1$.  Thus, members of $U \cap V$ may be obtained in a unique manner from compositions in $\mathcal{B}_{n-2}$ containing at least two $2$'s by inserting a 1 both directly before and after the rightmost 2.  This implies $|U\cap V|=y(b_{n-2}-n+2)$, upon excluding the $n-2$ members of $\mathcal{B}_{n-2}$ containing at most one $2$.  Next, observe $\pi \in U \cap W$ implies it is of the form $\pi=\pi'2^21^x$ for some $x\geq 1$.  Hence, $\pi$ may be formed from $\rho \in \mathcal{B}_{n-3}-\{1^{n-3}\}$ by inserting $2,1$ directly after the rightmost 2 of $\rho$, which leaves the capacity unchanged.  Therefore, we have $|U \cap W|=b_{n-3}-1$.  Recurrence \eqref{bnrece2} now follows from \eqref{incexcl}, upon substituting the expressions found above for the weights of the sets on the right-hand side.  \hfill \qed

\section{A further polynomial generalization}

In this section, we consider a polynomial generalization of $b_n$ in terms of a pair of statistics on $\mathcal{B}_n$ as follows.  Let $\tau(\pi)$ denote the number of $1$'s in $\pi \in \mathcal{B}_n$.  For our second statistic, let $1 \leq i_1<\cdots<i_r \leq m$ denote the set of indices $j \in [m]$ such that $\pi_j=2$ within $\pi=\pi_1\cdots \pi_m \in \mathcal{B}_n$, where $\pi$ contains $m$ parts for some $m \in [n]$.  Define $$\sigma (\pi)=r+\sum_{j=1}^r\sum_{k=1}^{i_j-1}\pi_k.$$

\noindent \textbf{Remarks:} We have that $\sigma(\pi)$ then gives the total of the number of parts of size 2 in $\pi$ and the sum of all the partial sums of $\pi$ consisting of the parts lying strictly to the left of some part of size 2.  Note that a composition in $\mathcal{B}_n$ may be regarded as a member of $\mathcal{L}_n$ by putting a square for each part 1 and a domino for each 2 and then writing the numbers $1,2,\ldots,n$ from left to right in the $n$ positions covered by the resulting tiling.  Thus, the analogous statistic on $\mathcal{L}_n$ (see \cite{SW}) records the sum of the numbers that are covered by the left halves of the dominoes within a tiling.  An equivalent statistic was considered earlier by Carlitz \cite{Car} recording the value of $a_1+2a_2+\cdots+(n-1)a_{n-1}$ on the set of $\{0,1\}$-words $a_1a_2\cdots a_{n-1}$ in which no two consecutive $1$'s occur.  Finally, the capacity statistic on $\mathcal{B}_n$ is seen to be equivalent to the statistic on $\mathcal{L}_n$ recording the total number of squares occurring between the first and last dominoes. \medskip

Define ${\bf b}_n={\bf b}_n(y;p,q)$ by
$${\bf b}_n=\sum_{\pi \in \mathcal{B}_n}p^{\tau(\pi)}q^{\sigma(\pi)}y^{\text{cap}(\pi)}, \qquad n \geq 1,$$
with ${\bf b}_0=1$.  Note that ${\bf b}_n(y;1,1)=b_n(y)$ for all $n$. For example, if $n=6$, then
$${\bf b}_6(y;p,q)=q^9+p^2q^4(1+q^2+q^4)+p^2q^5y(1+q^2+qy)+p^4q(1+q+q^2+q^3+q^4)+p^6,$$
which reduces to $b_6(y)=10+2y+y^2$ when $p=q=1$.

Define the generating function $F(x)=F(x,y;p,q)$ by
$$F(x)=\sum_{n\geq0}{\bf b}_nx^n.$$

\begin{theorem}\label{F(x,y;p,q)}
We have
\begin{equation}\label{F(x,y;p,q)e1}
F(x,y;p,q)=\frac{1}{1-px}+\sum_{n\geq1}\frac{q^{n^2}x^{2n}}{(1-px)(1-pq^nx)\prod_{j=1}^{n-1}(1-pq^jxy)}.
\end{equation}
\end{theorem}
\begin{proof}
Let ${\bf b}_n^{(i)}$ for $i=1,2$ denote the restriction of ${\bf b}_n$ to members of $\mathcal{B}_n$ whose last part is $i$.  Then we have
\begin{align}
{\bf b}_n^{(1)}&=p{\bf b}_{n-1}, \qquad n \geq 1,\label{bn(1)}\\
{\bf b}_n^{(2)}&=p^{n-2}q^{n-1}+\sum_{j=2}^{n-2}(py)^{j-2}q^{n-1}{\bf b}_{n-j}^{(2)}, \qquad n \geq 2,\label{bn(2)}
\end{align}
with ${\bf b}_0=1$, ${\bf b}_1^{(2)}=0$ and ${\bf b}_n={\bf b}_n^{(1)}+{\bf b}_n^{(2)}$ for $n \geq 1$. Let $F^{(i)}(x)=\sum_{n\geq1}{\bf b}_n^{(i)}x^n$ for $i=1,2$. Then \eqref{bn(1)} implies $F^{(1)}(x)=pxF(x)$ and $F(x)=1+F^{(1)}(x)+F^{(2)}(x)$, by the definitions.  Hence, we have
\begin{equation}\label{F(x,y;p,q)e2}
F(x)=\frac{1+F^{(2)}(x)}{1-px}.
\end{equation}

To determine $F^{(2)}(x)$, first note
\begin{align*}
\sum_{n\geq4}x^n\sum_{j=2}^{n-2}(py)^{j-2}q^{n-1}{\bf b}_{n-j}^{(2)}&=\frac{1}{q}\sum_{j\geq2}(py)^{j-2}\sum_{n\geq j+2}{\bf b}_{n-j}^{(2)}(qx)^n=\frac{1}{q}\sum_{j\geq2}(py)^{j-2}\cdot(qx)^jF^{(2)}(qx)\\
&=F^{(2)}(qx)\sum_{j\geq2}p^{j-2}q^{j-1}x^jy^{j-2}=\frac{qx^2}{1-pqxy}F^{(2)}(x).
\end{align*}
Multiplying both sides of \eqref{bn(2)} by $x^n$, and summing over all $n \geq 2$, then gives
\begin{align}
F^{(2)}(x)&=\sum_{n\geq2}p^{n-2}q^{n-1}x^n+\sum_{n\geq4}x^n\sum_{j=2}^{n-2}(py)^{j-2}q^{n-1}{\bf b}_{n-j}^{(2)}\notag\\
&=\frac{qx^2}{1-pqx}+\frac{qx^2}{1-pqxy}F^{(2)}(qx). \label{f(x,y;p,q)funceq}
\end{align}
Iterating the functional equation \eqref{f(x,y;p,q)funceq} an infinite number of times, where $x$ is assumed to be sufficiently small in absolute value, yields
\begin{align*}
F^{(2)}(x)&=\frac{qx^2}{1-pqx}+\frac{qx^2}{1-pqxy}\left(\frac{q^3x^2}{1-pq^2x}+\frac{q^3x^2}{1-pq^2xy}F^{(2)}(q^2x)\right)\\
&=\cdots=\frac{qx^2}{1-pqx}+\sum_{n\geq1}\frac{q^{2n+1}x^2}{1-pq^{n+1}x}\prod_{j=1}^n\left(\frac{q^{2j-1}x^2}{1-pq^jxy}\right)\\
&=\sum_{n\geq0}\frac{q^{(n+1)^2}x^{2n+2}}{(1-pq^{n+1}x)\prod_{j=1}^n(1-pq^jxy)}=\sum_{n\geq1}\frac{q^{n^2}x^{2n}}{(1-pq^{n}x)\prod_{j=1}^{n-1}(1-pq^jxy)},
\end{align*}
from which \eqref{F(x,y;p,q)e1} now follows from \eqref{F(x,y;p,q)e2}.
\end{proof}

Taking $q=1$ in \eqref{F(x,y;p,q)e1} gives the following formula for $F(x,y;p,1)$, which reduces to \eqref{Fxy} when $p=1$.

\begin{corollary}\label{F(x,y;p,q)cor}
We have
\begin{equation}\label{F(x,y;p,q)core1}
F(x,y;p,1)=\frac{1-px(1+y)+p^2x^2y+px^3(1-y)}{(1-px)^2(1-pxy-x^2)}.
\end{equation}
\end{corollary}

Let $\mathcal{B}_{n,k,j}$ denote the subset of $\mathcal{B}_{n,k}$ whose members contain exactly $j$ parts $1$.  Using \eqref{F(x,y;p,q)core1}, one can find an explicit formula for its cardinality.

\begin{theorem}\label{th1ex}
Let $n \geq j\geq k\geq0 $ with $n \equiv j\,(mod~2)$. If $k \geq 1$ and $n \geq j+4$, then
\begin{equation}\label{th1exe1}
|\mathcal{B}_{n,k,j}|=(j-k+1)\binom{\frac{n-j-4}{2}+k}{k}.
\end{equation}
If $n \geq j+2$, then $|\mathcal{B}_{n,0,j}|=j+1$, with $|\mathcal{B}_{n,0,n}|=1$ for all $n \geq0$.
\end{theorem}
\begin{proof}
One can give an algebraic proof as follows.  By \eqref{F(x,y;p,q)core1}, we have
$$F(x,y;p,1)=\frac{1}{1-px}+\frac{x^2}{(1-px)^2(1-x^2)}+\frac{px^5y}{(1-px)^2(1-x^2)(1-pxy-x^2)}.$$
Now observe
\begin{align*}
\frac{px^5y}{(1-px)^2(1-x^2)(1-pxy-x^2)}&=\frac{x^4}{(1-px)^2}\sum_{k\geq1}\frac{(pxy)^k}{(1-x^2)^{k+1}}\\
&=x^4\sum_{k\geq1}\frac{(pxy)^k}{(1-x^2)^{k+1}}\sum_{j\geq 1}j(px)^{j-1}\\
&=\sum_{k\geq1}\sum_{j \geq k}(j-k+1)x^{j+4}y^kp^j\cdot \frac{1}{(1-x^2)^{k+1}}\\
&=\sum_{r\geq0}\sum_{k\geq 1}\sum_{j \geq k}(j-k+1)\binom{r+k}{k}x^{2r+j+4}y^kp^j.
\end{align*}
Letting $n=2r+j+4$, and extracting the coefficient of $x^ny^kp^j$ in the last expression, yields \eqref{th1exe1}.  The $\frac{1}{1-px}$ term accounts for the $n=j\geq0,\,k=0$ case, whereas writing
$$\frac{x^2}{(1-px)^2(1-x^2)}=\frac{x^2}{1-x^2}\sum_{j\geq0}(j+1)(px)^j=\sum_{r\geq1}\sum_{j\geq0}(j+1)x^{2r+j}p^j$$
implies $|\mathcal{B}_{n,0,j}|=j+1$ for $n \geq j+2$.
\end{proof}

Let $F_n=F_n(p)$ denote the $n$-th Fibonacci polynomial defined by $F_n=pF_{n-1}+F_{n-2}$ for $n \geq 2$, with $F_0=1$ and $F_1=p$.  Note that $F_n(p)$ when $p$ is a positive integer is seen to enumerate colored members of $\mathcal{B}_n$ wherein each part of size $1$ is assigned one of $p$ colors. Theorem \ref{totcapth} can be generalized in terms of  colored compositions as follows.

\begin{theorem}\label{totcapthgen}
The total capacity of all members of $\mathcal{B}_n$ for $n \geq0$ in which each 1 is assigned one of $p$ colors is given by
\begin{equation}\label{totcapthgene1}
[x^n]\frac{\partial }{\partial y}F(x,y;p,1)\mid_{y=1}=\frac{p^2nF_n+2p(n+1)F_{n-1}}{p^2+4}-2pF_{n+1}+p^{n}(n+2p^2).
\end{equation}
\end{theorem}
\begin{proof}
By \eqref{F(x,y;p,q)core1}, we have
$$\frac{\partial }{\partial y}F(x,y;p,1)\mid_{y=1}=\frac{px}{(1-px-x^2)^2}-\frac{px(1-px+x^2)}{(1-px)^2(1-px-x^2)}=\frac{px^5}{(1-px)^2(1-px-x^2)^2}.$$
Generalizing the argument given above for \eqref{n+5tot}, and using the fact $\sum_{i=0}^np^{n-i}F_i=F_{n+2}-p^{n+2}$, we have
\begin{equation}\label{x^ncoeff}
[x^n]\left(\frac{1}{(1-px)^2(1-px-x^2)^2}\right)=\sum_{i=0}^{n+4}F_iF_{n-i+4}-2F_{n+6}+p^{n+4}(n+5+2p^2), \quad n \geq0.
\end{equation}
Note that the right side of \eqref{x^ncoeff} is seen to give the correct value of zero for $-5 \leq n \leq -1$.  Upon comparing the generating functions of both sides, we have
\begin{equation}\label{Fibpolconv}
\sum_{i=0}^{n-1}F_iF_{n-i-1}=\frac{pnF_n+2(n+1)F_{n-1}}{p^2+4}, \qquad n \geq0,
\end{equation}
where $F_{-1}=0$. Formula \eqref{totcapthgene1} now follows from computing $p[x^{n-5}]\left(\frac{1}{(1-px)^2(1-px-x^2)^2}\right)$ for $n \geq0$ using \eqref{x^ncoeff} and \eqref{Fibpolconv}.
\end{proof}

\noindent \textbf{Remarks:} Rewriting formula \eqref{th1exe1} using $r=\frac{n-j-2}{2}$, where $1 \leq r \leq \lfloor(n-k-2)/2\rfloor$, and then summing over $r$ yields \eqref{wnkform}.  Taking $p=1$ in \eqref{totcapthgene1}, and noting $L_{n+1}=F_n+2F_{n-1}$, yields \eqref{totcap}. In addition, it is possible to provide combinatorial arguments of \eqref{th1exe1} and \eqref{totcapthgene1} by modifying appropriately those given above for \eqref{wnkform} and \eqref{totcap}, respectively.  Moreover, one can explain combinatorially the Fibonacci polynomial identity \eqref{Fibpolconv} used in obtaining Theorem \ref{totcapthgen} by extending the tiling argument given in \cite[Identity~58]{BQ}.

\end{document}